\documentclass{article}
\usepackage{fullpage}

\usepackage[utf8]{inputenc}
\usepackage{amsthm,amssymb}
\usepackage[leqno]{amsmath}
\usepackage{thmtools,mathtools}
\usepackage{caption,subcaption}
\usepackage{epsfig,graphicx,graphics,color,tikz,tikz-cd,tcolorbox}
\usetikzlibrary{calc,decorations.pathmorphing,shapes}
\usepackage{calc}
\usepackage{enumerate,enumitem}
\usepackage[sort]{cite}
\usepackage{xspace}
\usepackage{bm}
\usepackage{csquotes}
\usepackage{hyperref}
\hypersetup{
colorlinks=true,       
linkcolor=blue,        
citecolor=magenta,         
filecolor=magenta,     
urlcolor=cyan,         
linktoc=all
}

\theoremstyle{plain}
\newtheorem{thm}{Theorem}
\newtheorem{lem}[thm]{Lemma}

\newtheorem{cl}[thm]{Claim}

\newtheorem{conj}{Conjecture}

\theoremstyle{definition}

\newtheorem{ex}[thm]{Example}

\newtheorem{rem}[thm]{Remark}

\def\final{0}  
\def\iflong{\iffalse}
\ifnum\final=0  
\newcommand{\knote}[1]{{\color{red}[{\tiny \textbf{Kristóf:} \bf #1}]\marginpar{\color{red}*}}}
\newcommand{\tnote}[1]{{\color{blue}[{\tiny \textbf{Tamás:} \bf #1}]\marginpar{\color{blue}*}}}
\newcommand{\bnote}[1]{{\color{green}[{\tiny \textbf{Bence:} \bf #1}]\marginpar{\color{green}*}}}
\newcommand{\anote}[1]{{\color{purple}[{\tiny \textbf{Áron:} \bf #1}]\marginpar{\color{purple}*}}}
\else 
\newcommand{\knote}[1]{}
\newcommand{\tnote}[1]{}
\newcommand{\bnote}[1]{}
\newcommand{\anote}[1]{}
\fi

\newcommand{\bR}{\mathbb{R}}
\newcommand{\bZ}{\mathbb{Z}}
\newcommand{\cB}{\mathcal{B}}

\newcommand{\cH}{\mathcal{H}}

\newcommand{\wC}{\widehat{C}}

\makeatletter
\renewcommand{\paragraph}{%
  \@startsection{paragraph}{4}%
  {\z@}{1.1ex \@plus 1ex \@minus .2ex}{-0.5em}%
  {\normalfont\normalsize\bfseries}%
}
\makeatother

\usepackage{titling}

\title{Cyclic ordering of split matroids}

\thanksmarkseries{alph}
\author{
Kristóf Bérczi\thanks{MTA-ELTE Matroid Optimization Research Group and HUN-REN–ELTE Egerváry Research Group, Department of Operations Research, Eötvös Loránd University, Budapest, Hungary. Email: \texttt{kristof.berczi@ttk.elte.hu}.}
\and
Áron Jánosik\thanks{MTA-ELTE Matroid Optimization Research Group, Department of Operations Research, Eötvös Loránd University, Budapest, Hungary. Email: \texttt{aron@janosik.hu}.}
\and
Bence Mátravölgyi\thanks{MTA-ELTE Matroid Optimization Research Group, Pázmány Péter sétány 1/C, H-1117, Budapest, Hungary and ETH Zurich, Rämistrasse 101, 8092, Zürich, Switzerland. Email: \texttt{bmatravoelgy@student.ethz.ch}} 
}

\date{}

\begin{document}
\maketitle

\begin{abstract}
There is a long list of open questions rooted in the same underlying problem: understanding the structure of bases or common bases of matroids. These conjectures suggest that matroids may possess much stronger structural properties than are currently known. One example is related to cyclic orderings of matroids. A rank-$r$ matroid is called cyclically orderable if its ground set admits a cyclic ordering such that any interval of $r$ consecutive elements forms a basis. In this paper, we show that if the ground set of a split matroid decomposes into pairwise disjoint bases, then it is cyclically orderable. This result answers a conjecture of Kajitani, Ueno, and Miyano in a special case, and also strengthens Gabow's conjecture for this class of matroids. Our proof is algorithmic, hence it provides a procedure for determining a cyclic ordering in question using a polynomial number of independence oracle calls. 

\end{abstract}
\section{Introduction}
\label{sec:intro}

Throughout the paper, we denote a matroid by $M=(S,\cB)$, where $S$ is a finite ground set and $\cB$ is the {\it family of bases}, satisfying the so-called {\it basis axioms}: (B1) $\emptyset\in\cB$, and (B2) for any $B_1,B_2\in\cB$ and $e\in B_1-B_2$, there exists $f\in B_2-B_1$ such that $B_1-e+f\in\cB$. The latter property, called the {\it basis exchange axiom}, is one of the most fundamental tools in matroid theory. Nevertheless, it only provides a local characterization of the relationship between bases, which presents a significant stumbling block to further progress.

A rank-$r$ matroid $M = (S, \cB)$ with $|S| = n$ is cyclically orderable if there exists an ordering $S = {s_1, \dots, s_n}$ such that ${s_i, s_{i+1}, \dots, s_{i+r-1}} \in \cB$ for every $i \in [n]$, where indices are understood in a cyclic order. While studying the structure of symmetric exchanges in matroids, Gabow~\cite{gabow1976decomposing} formulated a beautiful conjecture, stating that every matroid whose ground set decomposes into two disjoint bases is cyclically orderable. This question was raised independently by Wiedemann~\cite{wiedemann1984cyclic} and by Cordovil and Moreira~\cite{cordovil1993bases}. The conjecture makes a stronger claim: for a fixed partition, the cyclic ordering can be chosen such that the elements of the two bases in the partition form contiguous intervals.

\begin{conj}[Gabow] \label{conj:gabow}
Let $M=(S,\cB)$ be a matroid and $S=B_1\cup B_2$ be a partition of the ground set into two disjoint bases. Then, $M$ has a cyclic ordering in which the elements of $B_1$ and $B_2$ form intervals.
\end{conj}

It is not difficult to see that the statement holds for strongly base orderable matroids. The conjecture was settled for graphic matroids \cite{kajitani1988ordering,cordovil1993bases,wiedemann1984cyclic}, sparse paving matroids \cite{bonin2013basis}, matroids of rank at most $4$ \cite{kotlar2013serial} and $5$ \cite{kotlar2013circuits}, split matroids~\cite{berczi2024exchange}, and regular matroids~\cite{berczi2024reconfiguration}. However, the existence of a cyclic ordering remains open in general, even without the constraint of the bases forming intervals.

In~\cite{kajitani1988ordering}, Kajitani, Ueno, and Miyano proposed a conjecture that would provide a full characterization of cyclically orderable matroids. A matroid $M=(S,\cB)$ with {\it rank function} $r_M$ is called {\it uniformly dense} if $|S|\cdot r_M(X)\geq r_M(S)\cdot |X|$ holds for every $X\subseteq S$. It is not difficult to see that a cyclically orderable matroid is necessarily uniformly dense as well, and the conjecture states that this condition is also sufficient.

\begin{conj}[Kajitani, Ueno, and Miyano] \label{conj:cyclic}
A matroid is cyclically orderable if and only if it is uniformly dense.
\end{conj}

Despite the fact that the conjecture would provide entirely new insights into the structure of matroids, very little progress has been made so far. Van den Heuvel and Thomassé~\cite{van2012cyclic} showed that the conjecture is true if $|S|$ and $r(S)$ are coprimes, and Bonin's result~\cite{bonin2013basis} for sparse paving matroids remains true also in this more general setting. Recently, McGuinness~\cite{mcguinness2024cyclic} made a significant progress by verifying the conjecture for paving matroids.

It is worth taking a moment to consider the interpretation of the uniformly dense property. By the matroid union theorem of Edmonds and Fulkerson~\cite{edmonds1965transversals}, the ground set of a matroid $M=(S,\cB)$ can be covered by $k$ bases if and only if $k\cdot r_M(X)\geq |X|$ holds for every $X\subseteq S$. Using this, a matroid is uniformly dense if and only if its ground set can be covered by $\lceil |S|/r_M(S)\rceil $ bases. In other words, the ground set can be decomposed in `almost' disjoint bases, where almost means that the total overlapping between distinct bases is bounded by $r_M(S)-1$. In particular, any matroid whose ground set decomposes into pairwise disjoint bases is uniformly dense. This observation motivates the following strengthening of Gabow's conjecture.

\begin{conj}\label{conj:main}
Let $M=(S,\cB)$ be a matroid and $S=B_1\cup\dots\cup B_k$ be a partition of the ground set into $k$ pairwise disjoint bases. Then, $M$ has a cyclic ordering in which the elements of $B_i$ form an interval for each $i\in[k]$.
\end{conj}

To the best of our knowledge, Conjecture~\ref{conj:main} has not been previously considered and remains open even for very restricted classes of matroids, such as strongly base orderable matroids. Our main contribution is proving the conjecture for the class of split matroids. Split matroids were first introduced by Joswig and Schröter~\cite{joswig2017matroids} while studying matroid polytopes from a geometric point of view. Since then, this class of matroids has gained importance in many contexts, primarily due to the work of Ferroni and Schröter~\cite{ferroni2023enumerating,ferroni2023merino,ferroni2024tutte,ferroni2024valuative}. 

\begin{thm}\label{thm:main}
Conjecture~\ref{conj:main} is true for split matroids.
\end{thm}

It is worth emphasizing that our proof is algorithmic, hence it provides a procedure for determining a cyclic ordering in question using a polynomial number of independence oracle calls. 

\begin{rem}
In fact, we prove a slightly stronger statement: in the cyclic ordering obtained, the bases $B_1,\dots,B_k$ form intervals that follow each other in this order.   
\end{rem}

The rest of the paper is organized as follows. Basic definitions and notation are introduced in Section~\ref{sec:prelim}. We prove Conjecture~\ref{conj:main} for split matroids in Section~\ref{sec:split}. Finally, in Section~\ref{sec:open}, we give a list of related open questions and conjectures that are subject of future research.

\section{Preliminaries}
\label{sec:prelim}

\paragraph{General notation.} We denote the set of {\it nonnegative integers} by $\bZ_+$. For $k\in\bZ_+$, we use $[k]=\{1,\dots,k\}$. Given a ground set $S$, the \emph{difference} of $X,Y\subseteq S$ is denoted by $X-Y$. If $Y$ consists of a single element $y$, then $X-\{y\}$ and $X\cup \{y\}$ are abbreviated as $X-y$ and $X+y$, respectively. The {\it symmetric difference} of $X$ and $Y$ is denoted by $X\triangle Y\coloneqq (X-Y)\cup(Y-X)$.

\paragraph{Split matroids.} For basic definitions on matroids, we refer the reader to~\cite{oxley2011matroid}. Let $S$ be a ground set of size at least $r$, $\cH=\{H_1,\dots, H_q\}$ be a (possibly empty) collection of subsets of $S$, and $r, r_1, \dots, r_q$ be nonnegative integers satisfying 
\begin{align}
|H_i \cap H_j| &\le r_i + r_j -r\ \text{for distinct $i,j\in[q]$,}\tag*{(H1)}\label{eq:h1}\\
|S-H_i| + r_i &\ge r\ \text{for $i\in[q]$.} \tag*{(H2)}\label{eq:h2}
\end{align}
Then the corresponding {\it elementary split matroid} $M=(S,\cB)$ is given by $\cB=\{X\subseteq S\mid |X|= r,\ |X\cap H_i|\leq r_i\ \text{for $i\in[q]$}\}$; see \cite{berczi2023hypergraph} for details. It is easy to see that the underlying hypergraph can be chosen in such a way that 
\begin{align}
r_i &\le r-1\  \text{for $i\in[q]$,} \tag*{(H3)}\label{eq:h3}\\
|H_i| &\ge r_i+1\  \text{for $i\in[q]$.} \tag*{(H4)}\label{eq:h4}
\end{align}
The representation is called {\it non-redundant} if all of \ref{eq:h1}--\ref{eq:h4} hold. A set $F \subseteq S$ is called {\it $H_i$-tight} if $|F \cap H_i| = r_i$.  Finally, a {\it split matroid} is the direct sum of a single elementary split matroid and some (maybe zero) uniform matroids. The connection between elementary and connected split matroids is given by the following result~\cite{berczi2023hypergraph}.

\begin{lem}[B\'erczi, Kir\'aly, Schwarcz, Yamaguchi and Yokoi]\label{lem:esplit}
The classes of connected split matroids and connected elementary split matroids coincide.
\end{lem}

A nice feature of split matroids is that they generalize paving and sparse paving matroids: paving matroids correspond to the special case when $r_i=r-1$ for $i\in[q]$, while we get back the class of sparse paving matroids if, in addition, $|H_i|=r$ holds for $i\in[q]$. However, unlike the class of paving matroids, split matroids are closed not only under truncation and taking minors but also under duality~\cite{joswig2017matroids}. The following result appeared in~\cite{berczi2023hypergraph}.

\begin{lem}[B\'erczi, Kir\'aly, Schwarcz, Yamaguchi and Yokoi]\label{lem:tight}
Let $M$ be a rank-$r$ elementary split matroid with a non-redundant representation $\cH=\{H_1,\dots,H_q\}$ and $r,r_1,\dots,r_q$. Let $F$ be a set of size $r$.
\begin{enumerate}[label={(\alph*)}] \itemsep0em
\item If $F$ is $H_i$-tight for some index $i\in[q]$ then $F$ is a basis of $M$.\label{it:basis}
\item If $F$ is both $H_i$-tight and $H_j$-tight for distinct $i,j\in[q]$ then $H_i\cap H_j\subseteq F\subseteq H_i\cup H_j$.\label{it:f}
\end{enumerate} 
\end{lem}

By Lemma~\ref{lem:tight}\ref{it:basis}, any set of size $r$ that is tight with respect to one of the hyperedges is a basis. We will use this observation throughout without explicitly citing the lemma, to avoid repeatedly referring to part~\ref{it:basis}. 

\section{Proof of Theorem~\ref{thm:main}}
\label{sec:split}

\begin{proof}[Proof of Theorem~\ref{thm:main}]
The theorem clearly holds if $k=1$, while the case when $k=2$ was proved in~\cite{berczi2024exchange}. Therefore, we assume that $k\geq 3$.

Let $M=(S,\cB)$ be a split matroid and $S=B_1\cup\dots\cup B_k$ be a partition of its ground set into $k$ pairwise disjoint bases. First we show that it suffices to consider connected split matroids. To see this, let $M_1=(S_1,\cB_1),\dots,M_t=(S_t,\cB_t)$ be the connected components of $M$, where $|S_j|=n_j$ and the rank of $M_j$ is $r_j$ for $j\in[t]$. For $i\in[k]$ and $j\in[t]$, let $B^j_i\coloneqq B_i\cap S_j$. Then, $S_j=B^j_1\cup\dots\cup B^j_k$ is a decomposition of $S_j$ into pairwise disjoint bases of $M_j$. Let $S_j=\{s^j_1,\dots,s^j_{n_j}\}$ be a cyclic ordering of $M_j$ in which the elements of $B^j_i$ form the interval $I^j_i\coloneqq \{s^j_{(i-1)\cdot r_j+1},\dots,s^j_{i\cdot r_j}\}$ for each $i\in[k]$. Then, 
\begin{equation*}
S=\{I^1_1,I^2_1,\dots,I^t_1,I^1_2,I^2_2,\dots,I^t_2,\dots,I^1_k,I^2_k,\dots,I^t_k\}
\end{equation*}
is a cyclic ordering of $M$ in which $B_i$ forms an interval for each $i\in[k]$. Since Conjecture~\ref{conj:main} clearly holds for uniform matroids, the combination of the above observation and Lemma~\ref{lem:esplit} allows us to assume that $M$ is a rank-$r$ elementary split matroid, defined by a non-redundant representation $\cH$.

The high-level idea of the algorithm is as follows. We build up the orderings for the bases simultaneously in phases. At the beginning of the $j$-th phase, the first $(j-1)$ elements in each of the bases are ordered and the goal is to find the $j$-th element for all of them. We denote the first $(j-1)$ elements that we have already ordered in the $i$-th basis by $(b_{1}^i,\dots,b_{j-1}^{i})$. The elements that are not yet ordered will be referred to as {\it remaining elements} in $B_i$ and their set is denoted by $C_i$, that is, $C_i=B_i-\{b_{1}^i,\dots,b_{j-1}^{i}\}$. The goal is to choose $b_j^i$ in such a way that $(C_i-b_j^i) \cup (b_{1}^{i+1},\dots,b_j^{i+1})\ \text{ forms a basis for all}\ i\in[k]$\footnote{Throughout the proof, indices are meant in a cyclic order, i.e. $k+1\equiv 1$.}; we call such a choice $(b^1_j,\dots,b^k_j)$ {\it valid}. Note that the condition is satisfied in the beginning as it simply requires $C_i=B_i$ to be a basis for each $i\in[k]$. If valid choices exist up to the $r$-th phase, then we get a cyclic ordering of the matroid with the desired properties simply by putting the ordered bases after each other. However, if the next elements cannot be chosen while satisfying the above constraints, we will slightly modify the order of the first $(j-1)$ elements to allow further steps.

Now we turn to the detailed description of the proof. For ease of discussion, we present it as an indirect proof; however, it implicitly implies an algorithm as described above. Let $j\in[r+1]$ be maximal with respect to the property that, for $i\in[k]$, there exist $b_{1}^i,\dots,b_{j-1}^{i}\in B_i$ such that 
\begin{equation}
(b_{\ell}^{i},\dots,b_{j-1}^{i})\cup C_i \cup (b_{1}^{i+1},\dots,b_{\ell-1}^{i+1})\ \text{ forms a basis for all}\ i\in[k],\ell\in[j+1], \tag{$\star$} \label{eq:star}
\end{equation}
where $C_i=B_i-\{b_{1}^i,\dots,b_{j-1}^{i}\}$. If $j=r+1$ then we are done. Therefore, suppose that $j\leq r$. In particular, this means that there is no valid choice of $j$-th elements in the bases. Let $R_i\coloneqq C_i\cup \{b^{i+1}_1,\dots,b^{i+1}_{j-1}\}$ for $i\in[k]$. Then, $R_i$ is a basis by applying \eqref{eq:star} for $\ell=j+1$.

\begin{cl}\label{cl:p}
 For $i\in[k]$, there exist distinct elements $p_i,q_i\in C_i$ and a hyperedge $H_i$ with value $r_i$ satisfying the following:
\begin{enumerate}[label=(\alph*)]\itemsep0em
\item $p_i\in H_i$ for $i\in[k]$ and $p_i\in H_{i+1}$ for $i\in[k-1]$,
\item $q_i\notin H_i$ for $i\in[k]$ and $q_i\notin H_{i+1}$ for $i\in[k-1]$,
\item $p_k\notin H_1$,
\item $R_{i-1}$ is $H_i$-tight for $i\in[k]$.
\end{enumerate}
\end{cl}
\begin{proof}
Let $p_1\in C_1$ be an arbitrary element. By the basis exchange property for $R_1$ and $B_2$, there exists an element $p_2 \in C_2$ such that $R_1-p_1+p_2$ forms a basis. By the repeated application of this argument we get $p_i \in C_i$ such that $R_i-p_i+p_{i+1}$ forms a basis for $i\in[k-1]$. 

If $R_k-p_k+p_1$ forms a basis, then $(p_1,\dots,p_k)$ is a valid choice, contradicting the maximality of $j$. Otherwise, there exists a hyperedge $H_1$ with value $r_1$ such that $|H_1 \cap (R_k-p_k+p_1)| > r_1$. Since $R_k$ is a basis, we conclude that $R_k$ is $H_1$-tight, $p_k \notin H_1$, $p_1 \in H_1$ and $|H_1 \cap (R_k-p_k+p_1)| = r_1+1$. By the basis exchange property, there exists an element $q_1 \in C_1-p_1$ such that $R_k-p_k+q_1$ forms a basis, implying $q_1 \notin H_1$. As the choice $(q_1,p_2,\dots,p_k)$ cannot be valid, then there exists a hyperedge $H_2$ with value $r_2$ such that $|H_2 \cap (R_1-q_1+p_2)| > r_2$. Since $R_1$ and $R_1-p_1+p_2$ are both bases, we conclude that $R_1$ is $H_2$-tight, $p_1 \in H_2$, $p_2 \in H_2$, $q_1 \notin H_2$ and $|H_2 \cap (R_1-q_1+p_2)| = r_2+1$. By the basis exchange property, there exists an element $q_2 \in C_2-p_2$ such that $R_1-q_1+q_2$ forms a basis, implying $q_2 \notin H_2$. 
Continuing this procedure, we get elements $p_1,\dots,p_k$, $q_1,\dots,q_k$ and hyperedges $H_1,\dots,H_k$ with values $r_1,\dots,r_k$ satisfying the conditions of the claim.
\end{proof}

\begin{cl}\label{cl:q}
 For $i\in[k]$, there exist a hyperedge $H'_i$ with value $r'_i$ satisfying the following:
\begin{enumerate}[label=(\alph*)]\itemsep0em
\item $p_i\notin H'_i$ for $i\in[k]$ and $p_i\notin H'_{i+1}$ for $i\in[k-1]$,
\item $q_i\in H'_i$ for $i\in[k]$ and $q_i\in H'_{i+1}$ for $i\in[k-1]$,
\item $p_k\in H'_1$,
\item $q_k\in H_1$ and $q_k\notin H'_1$,
\item $R_{i-1}$ is $H'_i$-tight for $i\in[k]$,
\item $H_i\cap H'_i\subseteq R_{i-1}\subseteq H_i\cup H'_i$ for $i\in[k]$.
\end{enumerate}
\end{cl}
\begin{proof}
As the choice $(q_1,\dots,q_k)$ cannot be valid, there exists a hyperedge $H'_1$ with value $r'_{1}$ such that $q_k \notin H'_1$, $q_1 \in H'_1$ and $R_k$ is $H'_1$-tight. Note that since $q_1 \in H'_1$ and $q_1 \notin H_1$ we have that $H_1$ and $H'_1$ are distinct hyperedges. Since $R_k$ is both $H_1$ and $H'_1$-tight, Lemma~\ref{lem:tight}\ref{it:f} implies that $H_1\cap H'_1\subseteq R_k\subseteq H_1\cup H'_1$, thus $q_k \in H_1$, $p_k \in H'_1$ and $p_1 \notin H'_1$. As the choice $(p_1,q_2,\dots,q_k)$ cannot be valid and $R_k$ is $H_1$-tight by Claim~\ref{cl:p} and $q_k \in H_1$ therefore $R_k-q_k+p_1$ is also $H_1$-tight, there exists a hyperedge $H'_2$ with value $r'_2$ such that $p_1 \notin H'_2$, $q_2 \in H'_2$ and $R_1$ is $H'_2$-tight. Again as $q_2 \in H'_2$ and $q_2 \notin H_2$ we have that $H_2$ and $H'_2$ are distinct hyperedges. Since $R_1$ is both $H_2$ and $H'_2$-tight, Lemma~\ref{lem:tight}\ref{it:f} implies that $H_2\cap H'_2\subseteq R_1\subseteq H_2\cup H'_2$, thus $q_1 \in H'_2$, $p_2 \notin H'_2$. 
Continuing this procedure, we get hyperedges $H'_1,\dots,H'_k$ with values $r'_1,\dots,r'_k$ satisfying the conditions of the claim.
\end{proof}

\begin{cl}\label{cl:x}
For every $i\in[k-1]$ and $x \in C_i$, either $x \in (H_i \cap H_{i+1})-(H'_i \cup H'_{i+1})$ or $x \in (H'_i \cap H'_{i+1})-(H_i \cup H_{i+1})$. For $x\in C_k$, either $x \in (H_k \cap H'_1)-(H'_k \cup H_1)$ or $x \in (H'_k \cap H_1)-(H_k \cup H'_1)$.
\end{cl}
\begin{proof}
Consider any element $x \in C_i$ for some $i\in[k-1]$. By Claim~\ref{cl:q}, we know that $x\in H_{i+1}\cup H'_{i+1}$. We distinguish two cases based on which set $x$ belongs to.

Assume first that $x \in H_{i+1}$. If $x \notin H_i$, then $(q_1,\dots,q_{i-1},x,p_{i+1},\dots,p_k)$ is a valid choice for the set of $j$-th elements, contradicting the maximality of $j$. Indeed, $R_\ell$ is $H_{\ell+1}$-tight for every $\ell\in[k]$ by Claim~\ref{cl:p}, so $R_{i-1}-q_{i-1}+x$ and $R_i-x+p_{i+1}$ are $H_i$ and $H_{i+1}$-tight, respectively. Thus we have $x \in H_i$. By Claim~\ref{cl:q}, $H_i\cap H'_i\subseteq R_{i-1}$. As $x\notin R_{i-1}$, we have $x \notin H'_i$. If $x \in H'_{i+1}$, then $(p_1,\dots p_{i-1},x,q_{i+1},\dots,q_k)$ is a valid choice for the set of $j$-th elements, contradicting the maximality of $j$. Indeed, $R_\ell$ is $H'_{\ell+1}$ tight for every $\ell\in[k]$ by Claim~\ref{cl:q}, so $R_{i-1}-p_{i-1}+x$ and $R_i-x+q_{i+1}$ are $H'_i$ and $H'_{i+1}$-tight, respectively. Thus we have $x \notin H'_{i+1}$.
    
Consider now the case $x \in H'_{i+1}$. If $x \notin H'_i$, then $(p_1,\dots p_{i-1},x,q_{i+1},\dots,q_k)$ is a valid choice for the set of $j$-th elements, contradicting the maximality of $j$. Indeed,  $R_\ell$ is $H'_{\ell+1}$-tight for every $\ell\in[k]$ by Claim~\ref{cl:q}, so $R_{i-1}-p_{i-1}+x$ and $R_i-x+q_{i+1}$ are $H'_i$ and $H'_{i+1}$-tight, respectively. Thus we have $x\in H'_i$. By Claim~\ref{cl:q}, $H_i\cap H'_i\subseteq R_{i-1}$. As $x\notin R_{i-1}$, we have $x \notin H_i$. If $x \in H_{i+1}$, then $(q_1,\dots,q_{i-1},x,p_{i+1},\dots,p_k)$ is a valid choice for the set of $j$-th elements, contradicting the maximality of $j$. Indeed, $R_\ell$ is $H_{\ell+1}$ tight for every $\ell\in[k]$ by Claim~\ref{cl:q}, so $R_{i-1}-q_{i-1}+x$ and $R_i-x+p_{i+1}$ are $H_i$ and $H_{i+1}$-tight, respectively. This implies $x \notin H_{i+1}$.

Finally, the statement for $x\in C_k$ follows by replacing $H_{k+1}$ with $H'_1$ and $H'_{k+1}$ with $H_1$ in the argument above and using the right notion of tightness everywhere.
\end{proof}

For $i\in[k-1]$, we define $\wC_i\coloneqq \{x\in C_i\mid x \in (H_i \cap H_{i+1})-(H'_i \cup H'_{i+1})\}$ and $\wC'_i\coloneqq \{x\in C_i\mid x \in (H'_i \cap H'_{i+1})-(H_i \cup H_{i+1})\}$. We further set $\wC_k\coloneqq \{x\in C_k\mid x \in (H_k \cap H'_1)-(H'_k \cup H_1)\}$ and $\wC'_k\coloneqq \{x\in C_k\mid x \in (H'_k \cap H_1)-(H_k \cup H'_1)\}$. By Claim~\ref{cl:x}, $C_i=\wC_i\cup\wC'_i$ and $\wC_i\cap\wC'_i=\emptyset$ holds for each $i\in[k]$.

\begin{cl}\label{cl:card}
There exists an $s\in \bZ_+$ such that $|\wC_i|=|\wC'_i|=s$ for each $i\in[k]$.     
\end{cl}
\begin{proof}
As $B_2$ is a basis, we have $|B_2 \cap H_2| \leq r_2$. Since $|R_1 \cap H_2| = r_2$, we get $|\wC_2| = |C_2 \cap H_2| \leq |C_1 \cap H_2| = |\wC_1|$. A repeated application of the same argument leads to $|\wC_1| \geq |\wC_2|\geq \dots \geq |\wC_k|$. Similarly, as $B_1$ is a basis, we have $|B_1 \cap H'_1| \leq r'_{1}$. Since $|R_k \cap H'_1| = r'_1$, we get $|\wC'_1| = |C_1 \cap H'_1| \leq |C_k \cap H'_1| = |\wC_k|$. A repeated application of the same argument leads to $|\wC_k| \geq |\wC'_1| \geq |\wC'_2| \dots \geq |\wC'_k|$. Finally, as $B_1$ is a basis, we have $|B_1 \cap H_1| \leq r_{1}$. Since $|R_k \cap H_1| = r_1$, we get $|\wC_1| = |C_1 \cap H_1| \leq |C_k \cap H_1| = |\wC'_k|$. 

Concluding the above, we get $|\wC_1| \geq |\wC_2|\geq \dots \geq |\wC_k| \geq |\wC'_1| \geq |\wC'_2|\geq  \dots \geq |\wC'_k| \geq |\wC_1|$, finishing the proof of the claim.
\end{proof}

\begin{cl}\label{cl:tight}
For each $i\in[k]$, $B_i$ is both $H_i$ and $H'_i$-tight.
\end{cl}
\begin{proof}
Recall that $R_{i-1}$ is $H_{i}$-tight by Claim~\ref{cl:p}. Therefore, by Claim~\ref{cl:card}, we have
\begin{align*}
    |B_i\cap H_i|
    {}&{}=
    |(B_i\cap R_{i-1}) \cap H_i|+|(B_i-R_{i-1})\cap H_i|\\
    {}&{}=
    |(B_i\cap R_{i-1}) \cap H_i|+|\wC_i|\\
    {}&{}=
    |(B_i\cap R_{i-1}) \cap H_i|+|\wC_{i-1}|\\
    {}&{}=|R_{i-1}\cap H_i|\\
    {}&{}=
    r_i.
\end{align*}
Similarly, recall that $R_{i-1}$ is $H'_i$-tight by Claim~\ref{cl:q}. Therefore, by Claim~\ref{cl:card}, we have
\begin{align*}
    |B_i\cap H'_i|
    {}&{}=
    |(B_i\cap R_{i-1}) \cap H'_i|+|(B_i-R_{i-1})\cap H'_i|\\
    {}&{}=
    |(B_i\cap R_{i-1}) \cap H'_i|+|\wC'_i|\\
    {}&{}=
    |(B_i\cap R_{i-1}) \cap H'_i|+|\wC'_{i-1}|\\
    {}&{}=|R_{i-1}\cap H'_i|\\
    {}&{}=
    r'_i.
\end{align*}
This concludes the proof of the claim.
\end{proof}

By Claim~\ref{cl:p} we know that $p_k \in H_k$ and $p_k \notin H_1$ therefore $H_1 \neq H_k$. This means that there must exist consecutive indices $p$ and $p+1$ such that $H_p \neq H_{p+1}$. By definition, we know that $|H_p \cap H_{p+1}| \geq |\wC_p| = s$. Since $R_{p-1}$ is $H_p$-tight by Claim~\ref{cl:p} and none of the elements in $\wC'_{p-1}$ is in $H_p$, we get $|R_{p-1}|=r\geq r_p+s$. Since $R_p$ is $H_{p+1}$-tight by Claim~\ref{cl:p} and none of the elements in $\wC'_p$ is in $H_{p+1}$, we get $|R_p|=r\geq r_{p+1}+s$. In the case $p=1$ you need to replace $\wC'_{p-1}$ by $\wC_{k}$ and $R_{p-1}$ by $R_k$ in the argument above. These observations give
\begin{equation*}
s\leq|H_p\cap H_{p+1}|\leq r_p+r_{p+1}-r\leq (r-s)+(r-s)-r=r-2s,    
\end{equation*}
thus $s\leq r/3$. As $r=|B_i|=j-1+|C_i|=j-1+|\wC_i|+|\wC'_i|=j-1+2s\leq j-1+2r/3$, we get $j-1\geq r/3$. In particular, this means that at least one element is already ordered in each of $B_1,\dots,B_k$.\\

Now we turn our attention to the elements that have been already ordered. Consider the elements $b^i_t$ for $i\in[k]$, $t\in[j-1]$. Our goal is to show that the set of hyperedges containing these elements also have a specific structure.

\begin{cl}\label{cl:xxxxxxxx}
We have the following.
\begin{enumerate}[label=(\alph*)]\itemsep0em
\item For every $t\in[j-1]$, $b^i_{t}\in (H_i\triangle H'_i)\cap (H_{i+1}\triangle H'_{i+1})$ for every $i\in[k]$.\label{it:a}
\item For every $t\in[j-1]$, either $\{b^i_t,b^{i+1}_t\}\subseteq H_{i+1}$ or $\{b^i_t,b^{i+1}_t\}\subseteq H'_{i+1}$.\label{it:b}
\item For every $t\in[j-1]$, the set $\{b^i_t,\dots,b^i_{j-1}\} \cup C_i \cup \{b^{i+1}_1,\dots,b^{i+1}_{t-1}\}$ is $H_{i+1}$ and $H'_{i+1}$-tight.\label{it:c}
\end{enumerate}
\end{cl}
\begin{proof}
We prove the statement by induction on $t$ in a decreasing order. Assume that the statement is true for indices greater than $t$; when $t=j-1$, this assumption is meaningless. We first prove that \ref{it:a} holds for $t$. Consider an $i\in[k]$. As $b^i_t \in B_i$ and $B_i$ is $H_i$ and $H'_i$ tight by Claim~\ref{cl:tight}, we get that $b^i_t \in H_i \cup H'_i$. We first prove that $b^i_t \in (H_i\triangle H'_i)$ by showing that $b^i_t\in H_i\cap H'_i$ leads to a contradiction. As $b^i_t\notin R_i$, it is contained in at most one of $H_{i+1}$ and $H'_{i+1}$ by Claim~\ref{cl:q}, hence we consider two cases.\\

\noindent \textbf{Case 1.} $b^i_t\in H_i\cap H'_i$.

\medskip
\noindent \textbf{Case 1.1.} $b^i_t \notin H_{i+1}$.

Suppose first that $i<k$. Substitute $b^i_t$ with $q_i$ in the ordering of $B_i$. We claim that \eqref{eq:star} remains true. Indeed, by Claim~\ref{cl:p}, $q_i \notin H_{i+1}$ and as $\{b^i_{m},\dots,b^i_{j-1}\} \cup C_i \cup \{b^{i+1}_1,\dots,b^{i+1}_{m-1}\}$ is $H_{i+1}$-tight for every $t<m$ by induction, we get that $\{b^i_{m},\dots,b^i_{j-1}\} \cup C_i \cup \{b^{i+1}_1,\dots,b^{i+1}_{m-1}\}-q_i+b^i_t$ remains $H_{i+1}$-tight for every $t<m$. By Claim~\ref{cl:q}, $q_i \in H'_i$ and as $\{b^{i-1}_{m},\dots,b^{i-1}_{j-1}\} \cup C_{i-1} \cup \{b^{i}_1,\dots,b^{i}_{m-1}\}$ is $H'_i$-tight for every $t<m$ by induction, we get that $\{b^{i-1}_{m},\dots,b^{i-1}_{j-1}\} \cup C_{i-1} \cup \{b^{i}_1,\dots,b^{i}_{m-1}\}-b^i_t+q_i$ remains $H'_i$-tight for every $t<m$. After the modification, the choice $(q_1,\dots,q_{i-1},p_i,\dots,p_k)$ is valid for the $j$-th phase if $i\geq 2$. Indeed, by Claim~\ref{cl:p}, $q_{i-1} \notin H_{i}$, $p_i \in H_i$ and $q_i \notin H_i$, so $R_{i-1}-q_{i-1}-b^i_t+p_i+q_i$ is $H_i$-tight. For $i=1$, replace $q_{i-1}$ with $p_k$ and observe that $p_k \notin H_1$ by Claim~\ref{cl:p}. Also by Claim~\ref{cl:p}, $p_{i+1} \in H_{i+1}$, $p_i \in H_{i+1}$ and $q_i \notin H_{i+1}$, so $R_i-p_i-q_i+b^i_t+p_{i+1}$ remains $H_{i+1}$-tight. This contradicts the maximal choice of $j$.\\

Now consider the case $i=k$ and $b^k_t \notin H_{1}$. Substitute $b^k_t$ with $p_k$ in the ordering of $B_k$.  We claim that \eqref{eq:star} remains true. Indeed, by Claim~\ref{cl:p} , $p_k \notin H_1$ and as $\{b^k_{m},\dots,b^k_{j-1}\} \cup C_k \cup \{b^{1}_1,\dots,b^{1}_{m-1}\}$ is $H_1$-tight for every $t<m$ by induction, we get that $\{b^k_{m},\dots,b^k_{j-1}\} \cup C_k \cup \{b^{1}_1,\dots,b^{1}_{m-1}\}-p_k+b^k_t$ remains $H_1$-tight for every $t<m$. By Claim~\ref{cl:p}, $p_k \in H_k$ and as $\{b^{k-1}_{m},\dots,b^{k-1}_{j-1}\} \cup C_{k-1} \cup \{b^{k}_1,\dots,b^{k}_{m-1}\}$ is $H_k$-tight for every $t<m$ by induction, we get that $\{b^{k-1}_{m},\dots,b^{k-1}_{j-1}\} \cup C_{k-1} \cup \{b^{k}_1,\dots,b^{k}_{m-1}\}-b^k_t+p_k$ remains $H_k$-tight for every $t<m$. After the modification, the choice $(p_1,\dots,p_{k-1},q_k)$ is valid for the $j$-th phase. Indeed, by Claim~\ref{cl:q}, $p_{k-1} \notin H'_{k}$, $p_k \notin H'_k$ and $q_k \in H'_k$, so $R_{k-1}-p_{k-1}-b^k_t+p_k+q_k$ is $H'_k$-tight. By Claims~\ref{cl:p} and \ref{cl:q}, $p_1 \in H_1$, $p_k \notin H_1$ and $q_k \in H_1$ so $R_k-p_k-q_k+b^k_t+p_1$ remains $H_1$-tight. This contradicts the maximal choice of $j$.

\medskip
\noindent \textbf{Case 1.2.} $b^i_t \notin H'_{i+1}$.

Suppose first that $i<k$. Substitute $b^i_t$ with $p_i$ in the ordering of $B_i$. We claim that \eqref{eq:star} remains true. Indeed, by Claim~\ref{cl:q}, $p_i \notin H'_{i+1}$ and as $\{b^i_{m},\dots,b^i_{j-1}\} \cup C_i \cup \{b^{i+1}_1,\dots,b^{i+1}_{m-1}\}$ is $H'_{i+1}$-tight for every $t<m$ by induction, we get that $\{b^i_{m},\dots,b^i_{j-1}\} \cup C_i \cup \{b^{i+1}_1,\dots,b^{i+1}_{m-1}\}-p_i+b^i_t$ remains $H'_{i+1}$-tight for every $t<m$. By Claim~\ref{cl:p}, $p_i \in H_i$ and as $\{b^{i-1}_{m},\dots,b^{i-1}_{j-1}\} \cup C_{i-1} \cup \{b^{i}_1,\dots,b^{i}_{m-1}\}$ is $H_i$-tight for every $t<m$ by induction, we get that $\{b^{i-1}_{m},\dots,b^{i-1}_{j-1}\} \cup C_{i-1} \cup \{b^{i}_1,\dots,b^{i}_{m-1}\}-b^i_t+p_i$ remains $H_i$-tight for every $t<m$. After the modification, the choice $(p_1,\dots,p_{i-1},q_i,\dots,q_k)$ is valid for the $j$-th phase if $i\geq 2$. Indeed, by Claim~\ref{cl:q}, $p_{i-1} \notin H'_{i}$, $p_i \notin H'_i$ and $q_i \in H'_i$ so $R_{i-1}-p_{i-1}-b^i_t+p_i+q_i$ is $H'_i$-tight. For $i=1$, replace $p_{i-1}$ with $q_k$ and observe that $q_k \notin H'_1$ by Claim~\ref{cl:q}. Also by Claim~\ref{cl:q}, $q_{i+1} \in H'_{i+1}$, $p_i \notin H'_{i+1}$ and $q_i \in H'_{i+1}$, so $R_i-p_i-q_i+b^i_t+q_{i+1}$ remains $H'_{i+1}$-tight. This contradicts the maximal choice of $j$.\\

Now consider the case $i=k$ and $b^k_t \notin H'_{1}$. Substitute $b^k_t$ with $q_k$ in the ordering of $B_k$. We claim that \eqref{eq:star} remains true. Indeed, by Claim~\ref{cl:q}, $q_k \notin H'_1$ and as $\{b^k_{m},\dots,b^k_{j-1}\} \cup C_k \cup \{b^{1}_1,\dots,b^{1}_{m-1}\}$ is $H'_1$-tight for every $t<m$ by induction, we get that $\{b^k_{m},\dots,b^k_{j-1}\} \cup C_k \cup \{b^{1}_1,\dots,b^{1}_{m-1}\}-q_k+b^k_t$ remains $H'_1$-tight for every $t<m$. By Claim~\ref{cl:q}, $q_k \in H'_k$ and as $\{b^{k-1}_{m},\dots,b^{k-1}_{j-1}\} \cup C_{k-1} \cup \{b^{k}_1,\dots,b^{k}_{m-1}\}$ is $H'_k$-tight for every $t<m$ by induction, we get that $\{b^{k-1}_{m},\dots,b^{k-1}_{j-1}\} \cup C_{k-1} \cup \{b^{k}_1,\dots,b^{k}_{m-1}\}-b^k_t+q_k$ remains $H'_k$-tight for every $t<m$. After the modification, the choice $(q_1,\dots,q_{k-1},p_k)$ is valid for the $j$-th phase. Indeed, by Claim~\ref{cl:p}, $q_{k-1} \notin H_{k}$, $p_k \in H_k$ and $q_k \notin H_k$, so $R_{k-1}-q_{k-1}-b^k_t+p_k+q_k$ is $H_k$-tight. Also by Claim~\ref{cl:q}, $q_1 \in H'_1$, $p_k \in H'_1$ and $q_k \notin H'_1$ so $R_k-p_k-q_k+b^k_t+q_1$ remains $H'_1$-tight. This contradicts the maximal choice of $j$.\\

Summarizing the above, we get $b^i_t \in (H_i\triangle H'_i)$. We now prove that $b^i_t \in (H_{i+1}\triangle H'_{i+1})$. We know that $b^i_t \notin (H_{i+1}\cap H'_{i+1})$, so it suffices to show that $b^i_t \notin (H_{i+1}\cup H'_{i+1})$ leads to a contradiction. We consider two cases based on whether $b^i_t \in H_i - (H'_i\cup H_{i+1} \cup H'_{i+1})$ or $b^i_t \in H'_i - (H_i\cup H_{i+1}\cup H'_{i+1})$.\\

\noindent \textbf{Case 2.1.} $b^i_t \in H_i - (H'_i\cup H_{i+1} \cup H'_{i+1})$.

Suppose first that $i<k$. Substitute $b^i_t$ with $p_i$ in the ordering of $B_i$. We claim that \eqref{eq:star} remains true. Indeed, by Claim~\ref{cl:q}, $p_i \notin H'_{i+1}$ and as $\{b^i_{m},\dots,b^i_{j-1}\} \cup C_i \cup \{b^{i+1}_1,\dots,b^{i+1}_{m-1}\}$ is $H'_{i+1}$-tight for every $t<m$ by induction, we get that $\{b^i_{m},\dots,b^i_{j-1}\} \cup C_i \cup \{b^{i+1}_1,\dots,b^{i+1}_{m-1}\}-p_i+b^i_t$ remains $H'_{i+1}$-tight for every $t<m$. By Claim~\ref{cl:p}, $p_i \in H_i$ and as $\{b^{i-1}_{m},\dots,b^{i-1}_{j-1}\} \cup C_{i-1} \cup \{b^{i}_1,\dots,b^{i}_{m-1}\}$ is $H_i$-tight for every $t<m$ by induction, we get that $\{b^{i-1}_{m},\dots,b^{i-1}_{j-1}\} \cup C_{i-1} \cup \{b^{i}_1,\dots,b^{i}_{m-1}\}-b^i_t+p_i$ remains $H_i$-tight for every $t<m$. After the modification, the choice $(q_1,\dots,q_i,p_{i+1},\dots,p_k)$ is valid for the $j$-th phase if $i\geq 2$. Indeed, by Claim~\ref{cl:q}, $q_{i-1} \in H'_i$, $p_i \notin H'_i$ and $q_i \in H'_i$, so $R_{i-1}-q_{i-1}-b^i_t+p_i+q_i$ is $H'_i$-tight. For $i=1$, replace $q_{i-1}$ with $p_k$ and observe that $p_k \in H'_1$ by Claim~\ref{cl:q}. Also by Claim~\ref{cl:p}, $p_{i+1}\in H_{i+1}$, $p_i \in H_{i+1}$ and $q_i \notin H_{i+1}$, so $R_i-p_i-q_i+b^i_t+p_{i+1}$ remains $H_{i+1}$-tight. This contradicts the maximal choice of $j$.\\

Now consider the case $i=k$ and $b^k_t \in H_k - (H'_k\cup H_{1} \cup H'_{1})$. Substitute $b^k_t$ with $p_k$ in the ordering of $B_k$. We claim that \eqref{eq:star} remains true. Indeed, by Claim~\ref{cl:p}, $p_k \notin H_1$ and as $\{b^k_{m},\dots,b^k_{j-1}\} \cup C_k \cup \{b^{1}_1,\dots,b^{1}_{m-1}\}$ is $H_1$-tight for every $t<m$ by induction, we get that $\{b^k_{m},\dots,b^k_{j-1}\} \cup C_k \cup \{b^{1}_1,\dots,b^{1}_{m-1}\}-p_k+b^k_t$ remains $H_1$-tight for every $t<m$. By Claim~\ref{cl:p}, $p_k \in H_k$ and as $\{b^{k-1}_{m},\dots,b^{k-1}_{j-1}\} \cup C_{k-1} \cup \{b^{k}_1,\dots,b^{k}_{m-1}\}$ is $H_k$-tight for every $t<m$ by induction, we get that $\{b^{k-1}_{m},\dots,b^{k-1}_{j-1}\} \cup C_{k-1} \cup \{b^{k}_1,\dots,b^{k}_{m-1}\}-b^k_t+p_k$ remains $H_k$-tight for every $t<m$. After the modification, the choice $(q_1,\dots,q_k)$ is valid for the $j$-th phase. Indeed, by Claim~\ref{cl:p}, $q_{k-1} \notin H_k$, $p_k \in H_k$ and $q_k \notin H_k$, so $R_{k-1}-q_{k-1}-b^k_t+p_k+q_k$ is $H_{k}$-tight. Also by Claim~\ref{cl:q}, $q_1 \in H'_1$, $q_k \notin H'_1$ and $p_k \in H'_1$, so $R_k-p_k-q_k+b^k_t+q_1$ remains $H'_1$-tight. This contradicts the maximal choice of $j$.

\medskip
\noindent \textbf{Case 2.2.} $b^i_t \in H'_i - (H_i\cup H_{i+1} \cup H'_{i+1})$.

Suppose first that $i<k$. Substitute $b^i_t$ with $q_i$ in the ordering of $B_i$. We claim that \eqref{eq:star} remains true. Indeed, by Claim~\ref{cl:p}, $q_i \notin H_{i+1}$ and as $\{b^i_{m},\dots,b^i_{j-1}\} \cup C_i \cup \{b^{i+1}_1,\dots,b^{i+1}_{m-1}\}$ is $H_{i+1}$-tight for every $t<m$ by induction, we get that $\{b^i_{m},\dots,b^i_{j-1}\} \cup C_i \cup \{b^{i+1}_1,\dots,b^{i+1}_{m-1}\}-q_i+b^i_t$ remains $H_{i+1}$-tight for every $t<m$. By Claim~\ref{cl:q}, $q_i \in H'_i$ and as $\{b^{i-1}_{m},\dots,b^{i-1}_{j-1}\} \cup C_{i-1} \cup \{b^{i}_1,\dots,b^{i}_{m-1}\}$ is $H'_i$-tight for every $t<m$ by induction, we get that $\{b^{i-1}_{m},\dots,b^{i-1}_{j-1}\} \cup C_{i-1} \cup \{b^{i}_1,\dots,b^{i}_{m-1}\}-b^i_t+q_i$ remains $H'_i$-tight for every $t<m$. After the modification, the choice $(p_1,\dots,p_i,q_{i+1},\dots,q_k)$ is valid for the $j$-th phase if $i\geq 2$. Indeed, by Claim~\ref{cl:p}, $p_{i-1} \in H_i$, $p_i \in H_i$ and $q_i \notin H_i$, so $R_{i-1}-p_{i-1}-b^i_t+p_i+q_i$ is $H_i$-tight. For $i=1$, replace $p_{i-1}$ with $q_k$ and observe that $q_k \in H_1$ by Claim~\ref{cl:q}. Also by Claim~\ref{cl:q}, $q_{i+1}\in H'_{i+1}$, $p_i \notin H'_{i+1}$ and $q_i \in H'_{i+1}$, so $R_i-p_i-q_i+b^i_t+q_{i+1}$ remains $H'_{i+1}$-tight. This contradicts the maximal choice of $j$.\\

Now consider the case $i=k$ and $b^k_t \in H'_k - (H_k\cup H_{1} \cup H'_{1})$. Substitute $b^k_t$ with $q_k$ in the ordering of $B_k$. We claim that \eqref{eq:star} remains true. Indeed, by Claim~\ref{cl:q}, $q_k \notin H'_1$ and as $\{b^k_{m},\dots,b^k_{j-1}\} \cup C_k \cup \{b^{1}_1,\dots,b^{1}_{m-1}\}$ is $H'_1$-tight for every $t<m$ by induction, we get that $\{b^k_{m},\dots,b^k_{j-1}\} \cup C_k \cup \{b^{1}_1,\dots,b^{1}_{m-1}\}-q_k+b^k_t$ remains $H'_1$-tight for every $t<m$. By Claim~\ref{cl:q}, $q_k \in H'_k$ and as $\{b^{k-1}_{m},\dots,b^{k-1}_{j-1}\} \cup C_{k-1} \cup \{b^{k}_1,\dots,b^{k}_{m-1}\}$ is $H'_k$-tight for every $t<m$ by induction, we get that $\{b^{k-1}_{m},\dots,b^{k-1}_{j-1}\} \cup C_{k-1} \cup \{b^{k}_1,\dots,b^{k}_{m-1}\}-b^k_t+q_k$ remains $H'_k$-tight for every $t<m$. After the modification, the choice $(p_1,\dots,p_k)$ is valid for the $j$-th phase. Indeed, by Claim~\ref{cl:q}, $p_{k-1} \notin H'_k$, $p_k \notin H'_k$ and $q_k \in H'_k$, so $R_{k-1}-p_{k-1}-b^k_t+p_k+q_k$ is $H'_{k}$-tight. Also by Claims~\ref{cl:p} and \ref{cl:q}, $p_1 \in H_1$, $q_k \in H_1$ and $p_k \notin H_1$, so $R_k-p_k-q_k+b^k_t+p_1$ remains $H_1$-tight. This contradicts the maximal choice of $j$.\\

This finishes the proof of \ref{it:a}, that is, $b^i_{t}\in (H_i\triangle H'_i)\cap (H_{i+1}\triangle H'_{i+1})$. To prove the remaining two properties, observe that $\{b^i_{t+1},\dots,b^i_{j-1}\} \cup C_i \cup \{b^{i+1}_1,\dots,b^{i+1}_{t}\}-b^{i+1}_{t}+b^i_t$ is a basis by \eqref{eq:star}. As $\{b^i_{t+1},\dots,b^i_{j-1}\} \cup C_i \cup \{b^{i+1}_1,\dots,b^{i+1}_{t}\}$ is both $H_{i+1}$ and $H'_{i+1}$-tight, together with $b^i_t \in (H_{i+1}\triangle H'_{i+1})$ and $b^{i+1}_t \in (H_{i+1}\triangle H'_{i+1})$, necessarily $\{b^i_t,b^{i+1}_t\}\subseteq H_{i+1}$ or $\{b^i_t,b^{i+1}_t\}\subseteq H'_{i+1}$ as otherwise the basis would have too large intersection with $H_{i+1}$ or $H'_{i+1}$. This implies 
\begin{equation*}
    |(\{b^i_{t+1},\dots,b^i_{j-1}\} \cup C_i \cup \{b^{i+1}_1,\dots,b^{i+1}_{t}\}) \cap H_{i+1}| = |(\{b^i_{t+1},\dots,b^i_{j-1}\} \cup C_i \cup \{b^{i+1}_1,\dots,b^{i+1}_{t}\}-b^{i+1}_{t}+b^i_t)\cap H_{i+1}|
\end{equation*} 
and 
\begin{equation*}
|(\{b^i_{t+1},\dots,b^i_{j-1}\} \cup C_i \cup \{b^{i+1}_1,\dots,b^{i+1}_{t}\}) \cap H'_{i+1}| = |(\{b^i_{t+1},\dots,b^i_{j-1}\} \cup C_i \cup \{b^{i+1}_1,\dots,b^{i+1}_{t}\}-b^{i+1}_{t}+b^i_t)\cap H'_{i+1}|,
\end{equation*} 
which means that properties \ref{it:b} and \ref{it:c} hold as well.
\end{proof}

Claim~\ref{cl:tight} and \ref{cl:xxxxxxxx} imply that $B_i$ is tight with respect to $H_i$, $H'_i$, $H_{i+1}$ and $H'_{i+1}$. We know that $H_m \neq H_{m+1}$ for some $m<k$. Then, $B_m \subseteq H_m \cup H_{m+1}$ which implies $q_m \in B_m \subseteq H_m \cup H_{m+1}$. However, by Claim~\ref{cl:p}, $q_m \notin H_m \cup H_{m+1}$, a contradiction. This concludes the proof of the theorem.
\end{proof}

\section{Further remarks and open problems}
\label{sec:open}

\subsection{Comments on Conjecture~\ref{conj:cyclic}}

The most important result toward verifying Conjecture~\ref{conj:cyclic} is due to Van den Heuvel and Thomassé~\cite{van2012cyclic}. 

\begin{thm} [Van den Heuvel and Thomassé] \label{thm:cyclic}
Let $M=(S,\cB)$ be a loopless matroid with rank function $r\colon 2^S\to\bZ_+$ and $|S|=n$, and let $g$ denote the greatest common divisor of $r(S)$ and $n$. Then, there exists a partition $S=G_1\cup\dots\cup G_{n/g}$ into sets of size $g$ such that $\bigcup_{t=0}^{r(S)/g-1} G_{i+t}$ is a basis for $i\in[n/g]$ if and only if $r(S)\cdot |X|\leq n\cdot r(X)$ for $X\subseteq S$.
\end{thm}

In particular, Theorem~\ref{thm:cyclic} settles Conjecture~\ref{conj:cyclic} in the affirmative if $r(S)$ and $n$ are coprimes. Therefore, to prove Conjecture~\ref{conj:cyclic}, it would be enough to verify that, when $M$ is uniformly dense, the elements inside each $G_i$ admit an ordering that together induces a cyclic ordering of $M$. Unfortunately, such an approach cannot work as shown by the following example.

\begin{ex}
Let $S=\{a_1,\dots,a_{10}\}$ and consider the sparse paving matroid defined by the following hyperedges: $\{a_1,a_2,a_3,a_{10}\}$, $\{a_1,a_2,a_4,a_9\}$, $\{a_1,a_3,a_4,a_5\}$, $\{a_2,a_3,a_4,a_6\}$, $\{a_3,a_5,a_6,a_7\}$, $\{a_4,a_5,a_6,a_8\}$, $\{a_5,a_7,a_8,a_9\}$, $\{a_6,a_7,a_8,a_{10}\}$,$\{a_1,a_7,a_9,a_{10}\}$, $\{a_2,a_8,a_9,a_{10}\}$, with the value of $r$ being $4$; see Section~\ref{sec:prelim} for the definition. If $G_i=\{a_{2i-1},a_{2i}\}$ for $i\in[5]$, then it is not difficult to check that $G_i\cup G_{i+1}$ is a basis for every $i\in[10]$. 

However, we claim that the pairs in the sets $G_i$ cannot be ordered in such a way that we get a cyclic ordering of the matroid $M$. Indeed, each $G_i$ is contained in two of the hyperedges, which excludes two of the four possible orderings of the neighboring groups $G_{i-1}$ and $G_{i+1}$. Due to the exclusion of these ordering possibilities, it is not difficult to verify that no suitable ordering exists.
\end{ex}

\subsection{Exchange distance of basis sequences}

Note that Gabow's conjecture can be interpreted as follows: for any two disjoint bases $B_1$ and $B_2$ of a matroid $M$ of rank $r$, there is a sequence of $r$ symmetric exchanges that transforms the pair $(B_1,B_2)$ into $(B_2,B_1)$. The closely related problem of transforming a sequence $(B_1,\dots, B_k)$ of bases into another $(B'_1,\dots,B'_k)$ was proposed by White~\cite{white1980unique}. Let $(B_1,\dots,B_k)$ be a sequence  of $k$ bases of a matroid $M$, and assume that there exist $e\in B_i$, $f\in B_j$ for some $1\leq i<j\leq k$ such that both $B_i-e+f$ and $B_j-f+e$ are bases. Then we say that the sequence $(B_1,\dots,B_{i-1},B_i-e+f,B_{i+1},\dots,B_{j-1},B_j-f+e,B_{j+1},\dots,B_k)$ is obtained from the original one by a \emph{symmetric exchange}. Accordingly, two sequences of bases are called \emph{equivalent} if one can be obtained from the other by a composition of symmetric exchanges. White studied the following question: what is the characterization of two sequences of bases being equivalent?

There is an easy necessary condition. Namely, two sequences $(B_1,\dots,B_k)$ and $(B'_1,\dots,B'_k)$ are called \emph{compatible} if the union of the $B_i$s as a multiset coincides with the union of the $B'_i$s as a multiset. Compatibility is obviously a necessary condition for two sequences being equivalent, and  White conjectured that it is also sufficient.

\begin{conj}[White] \label{conj:white}
Two sequences of $k$ bases are equivalent if and only if they are compatible.
\end{conj}

In this context, Gabow's conjecture would verify White's conjecture for two pairs of bases of the form $(B_1,B_2)$ and $(B_2,B_1)$. Note that, however, the conjecture says nothing on the minimum number of exchanges needed to transform one of the pairs into the other. As a common generalization of Gabow's conjecture and the special case of White's conjecture when $k=2$, Hamidoune~\cite{cordovil1993bases} proposed an optimization variant.

\begin{conj}[Hamidoune] \label{conj:hamidoune}
Let $(B_1,B_2)$ and $(B'_1,B'_2)$ be compatible basis pairs of a rank-$r$ matroid $M=(S,\cB)$. Then, $(B_1,B_2)$ can be transformed into $(B'_1,B'_2)$ by using at most $r$ symmetric exchanges.
\end{conj}

In~\cite{berczi2024weighted}, Bérczi, Mátravölgyi and Schwarcz formulated a weighted extension of Hamidoune's conjecture. Let $M=(S,\cB)$ be a matroid and $w\colon S\to\mathbb{R}_+$ be a weight function on the elements of the ground set $S$. Given a pair $(B_1,B_2)$ of bases, we define the {\it weight of a symmetric exchange} $B_1-e+f$ and $B_2-f+e$ to be $w(e)/2+w(f)/2$, that is, the average of the weights of the exchanged elements.

\begin{conj}[Bérczi, Mátravölgyi, Schwarcz]\label{conj:weighted}
Let $(B_1,B_2)$ and $(B'_1,B'_2)$ be compatible basis pairs of a matroid $M=(S,\cB)$, and let $w\colon S\to\bR_+$. Then, $(B_1,B_2)$ can be transformed into $(B'_1,B'_2)$ by using symmetric exchanges of total weight at most $w(B_1)/2+w(B_2)/2=w(B'_1)/2+w(B'_2)/2$.
\end{conj}

By setting the weights to be identically $1$, we get back Hamidoune's conjecture. The question naturally arises: can we formulate extensions of Conjectures~\ref{conj:hamidoune} and~\ref{conj:weighted} for basis sequences of length greater than two? 

Let $(B_1,\dots,B_k)$ be a sequence  of $k$ bases of a matroid $M$, and assume that there exists distinct indices $\{i_1,\dots,i_q\}\subseteq [k]$ and $e_j\in B_{i_j}$ such that $B_{i_j}-e_j+e_{j+1}$ is a basis for each $j\in[q]$. Then, we say that the sequence $(B'_1,\dots,B'_k)$ where $B'_{\ell}=B_{i_j}-e_{j}+e_{j+1}$ if $\ell=i_j$ for some $j\in[q]$ and $B'_\ell=B_\ell$ otherwise, is obtained by a {\it cyclic exchange}. As a generalization of Conjecture~\ref{conj:hamidoune}, we propose the following.

\begin{conj}
Let $(B_1,\dots,B_k)$ and $(B'_1,\dots,B'_k)$ be compatible sequences of $k$ bases of a rank-$r$ matroid. Then, $(B_1,\dots,B_k)$ can be transformed into $(B'_1,\dots,B'_k)$ by using at most $r$ cyclic exchanges.
\end{conj}

Given a weight function $w\colon S\to\mathbb{R}_+$ on the elements of the ground set, let us define the {\it weight of a cyclic exchange} that moves elements $e_j\in B_{i_j}$ for $j\in[q]$ to be $\frac{1}{k}\sum_{j=1}^q w(e_j)$. As a generalization of Conjecture~\ref{conj:weighted}, the weighted counterpart is as follows.

\begin{conj}
Let $(B_1,\dots,B_k)$ and $(B'_1,\dots,B'_k)$ be compatible sequences of $k$ bases of a matroid $M=(S,\cB)$, and let $w\colon S\to\bR_+$. Then, $(B_1,\dots,B_k)$ can be transformed into $(B'_1,\dots,B'_k)$ by using cyclic exchanges of total weight at most $\frac{1}{k}\sum_{i=1}^k w(B_i)=\frac{1}{k}\sum_{i=1}^k w(B'_i)$.
\end{conj}

Note that in both cases, the bounds are tight in the sense that $r$ cyclic exchanges are definitely needed to transform the sequence $(B_1,\dots,B_{k-1},B_k)$ into $(B_2,\dots,B_k,B_1)$.

\medskip

\paragraph{Acknowledgement} 

The authors are grateful to Tamás Schwarcz for helpful discussions. Áron Jánosik was supported by the EKÖP-24 University Excellence Scholarship Program of the Ministry for Culture and Innovation from the source of the National Research, Development and Innovation Fund. The research has been implemented with the support provided by the Lend\"ulet Programme of the Hungarian Academy of Sciences -- grant number LP2021-1/2021, by the Ministry of Innovation and Technology of Hungary -- grant number ELTE TKP 2021-NKTA-62, and by Dynasnet European Research Council Synergy project -- grant number ERC-2018-SYG 810115.

\bibliographystyle{abbrv}
\bibliography{gengab}

\end{document}